\documentclass{article}

\usepackage{amsmath}
\usepackage{amssymb}
\usepackage{amsthm}
\usepackage[latin1]{inputenc}

\usepackage{graphicx}
\usepackage{epstopdf}

\usepackage{color}

\usepackage{ifthen}

\renewcommand{\div}{\operatorname{div}}

\newcommand{\Tt}{{\mathbb{T}}}

 \newcommand{\Rr}{\mathbb R}

 \newcommand{\af}{\alpha}
 
 \newcommand{\ep}{\epsilon}
\newcommand{\be}{\beta}

\renewcommand{\div}{\operatorname{div}}

\newcommand{\tr}{\operatorname{Tr}}

\newtheorem{teo}{Theorem}[section]

\newtheorem{Lemma}{Lemma}[section]
\newtheorem{Corollary}{Corollary}[section]
\newtheorem{Proposition}{Proposition}[section]

\newtheorem{Assumption}{A}

\begin{document}

\title{Local regularity for mean-field games in the whole space}
\author{Diogo A. Gomes\footnote{King Abdullah University of Science and Technology (KAUST), CEMSE Division and
KAUST SRI, Uncertainty Quantification Center in Computational Science and Engineering, Thuwal 23955-6900. Saudi Arabia. e-mail: diogo.gomes@kaust.edu.sa.}, 
Edgard Pimentel
\footnote{Instituto Nacional de Matem\'atica Pura e Aplicada, IMPA. Estrada Dona Castorina, 110, 22460-320 Rio de Janeiro-RJ,
Brazil. e-mail: edgardap@impa.br.}
} 

\date{\today} 

\maketitle

\begin{abstract}
In this paper, we investigate the Sobolev regularity for mean-field games in the whole space $\Rr^d$. This is achieved by combining integrability for the solutions of the Fokker-Planck equation with estimates for the Hamilton-Jacobi equation in Sobolev spaces. To avoid the mathematical challenges due to the lack of compactness,
we prove an entropy dissipation estimate for the adjoint variable. This, together with the non-linear adjoint method, yields uniform estimates for solutions of the Hamilton-Jacobi equation in $W^{1,p}_{loc}(\Rr^d)$.
\end{abstract}

\thanks{
D. Gomes was partially supported by KAUST baseline funds, KAUST SRI, Uncertainty Quantification Center in Computational Science and Engineering and CAMGSD-LARSys through FCT-Portugal.

E. Pimentel is financed by CNPq-Brazil, grant 401795/2013-6.}

\section{Introduction}

The mathematical theory of mean-field games (MFG) formalizes the concept of Nash equilibrium, for $N$-players stochastic differential games, when $N\to\infty$. It comprises a variety of methods and techniques, which aim at investigating problems with a large number of agents. 
Introduced in the works of J-M. Lasry and P-L. Lions \cite{ll1, ll2, ll3} and M. Huang, P. Caines and R. Malham\'e, \cite{C1, C2}, these methods have known an intense research activity. Indeed, several research directions have been undertaken by various authors, see, for instance, the surveys \cite{llg2}, \cite{cardaliaguet}, \cite{achdou2013finite}, or \cite{GS}, as well as the lectures by P-L. Lions \cite{LCDF,LIMA} and the recent book by A. Bensoussan, J. Frehse and P. Yam \cite{bensoussan}.

In the present paper, we study a time-dependent MFG in the whole space $\Rr^d$. 
This MFG is defined through the following system of two partial differential equations 
in $\Rr^d\times [0, T]$:
\begin{equation}\label{mfg}
\begin{cases}
-u_t+H(x,Du)=\Delta u+g(m)\\m_t-\div(D_pH m)=\Delta m,
\end{cases}
\end{equation}
where $T>0$ is arbitrarily fixed and $u, m:\Rr^d\times [0,T]\to  \Rr$ satisfy 
the initial-terminal conditions
\begin{equation}\label{itbc}
\begin{cases}
u(x,T)=u_T(x)\\m(x,0)=m_0(x).
\end{cases}
\end{equation}
The Hamiltonian $H$ and the non-linearity $g$ satisfy a series of assumptions detailed in Section \ref{assumptions}. To illustrate our results, 
we observe that 
a model Hamiltonian for which these  hold is $$H(x,p)\,=\,a(x)\left(1+|p|^2\right)^\frac{\gamma}{2}+V(x),$$ where $\gamma\in\left(1,2\right]$ and $a,\,V\in\mathcal{C}^\infty$ and $V\leq 0$ is bounded. A typical non-linearity $g$ would be required to be non-decreasing and to satisfy
\begin{equation*}
g(m)\,=\,
\begin{cases}
m(x,t),&\;m<<1\\
m^\alpha(x,t),&\;m>> 1,
\end{cases}
\end{equation*}
for $\af>0$, interpolating in a smooth monotone way for $m$ near $1$ (see Assumption A\ref{ge}).

Existence of solutions for MFG is a matter of fundamental interest. Most of the results in the literature were investigated in the periodic setting or for smooth bounded domains, under Dirichlet or Neumann boundary conditions, see \cite{porretta2}.
In the stationary  periodic  setting, existence of weak solutions was obtained in \cite{ll1}. Smooth solutions were addressed in \cite{GM}, \cite{GPM1} and \cite{GPatVrt}
(see also \cite{GIMY}).
The stationary obstacle problem was investigated in \cite{GPat}, and the congestion problem in \cite{GMit}.
For the time-dependent case, weak solutions were considered in \cite{ll2},
\cite{porretta2}, and \cite{cgbt}. 
he planning problem was studied in \cite{porretta}.
In \cite{CLLP}, the authors have proven the existence of smooth solutions for quadratic Hamiltonians. In
\cite{LIMA}, the author obtained existence of classical solutions for \eqref{mfg}-\eqref{itbc}  under quadratic or subquadratic hypothesis. These results  were substantially improved in \cite{GPM2} (see also \cite{GPim2} for the case of logarithmic non-linearities).
Mean-field games with superquadratic Hamiltonians were investigated in \cite{GPM3}. 

Except for the explicit linear-quadratic models \cite{MR2928378},  regularity of solutions for MFG in the whole space has not been investigated in the literature before the present paper. A main difficulty is the absence of compactness of the domain. This entails several mathematical challenges, as various standard estimates in regularity theory for MFG are simply not valid. One of the key issues is that the Hamiltonian $H$ is no longer integrable. Because of this, the adjoint method as applied in \cite{GPM3} does not yield bounds for the Hamilton-Jacobi equation in terms of Lebesgue norms of the non-linearity $g$. 
To overcome this difficulty, we investigate the integrability of the adjoint variable. First, we prove an entropy dissipation estimate.
Thanks to this, we are then able to achieve local regularity in Sobolev spaces for the Hamilton-Jacobi equation in terms of $L^p(0,T;L^q(\Rr^d))$-norms of $g$. It is important to stress that the key novelty in this paper is to use this entropy dissipation estimate coupled with the adjoint method to obtain estimates in Sobolev spaces.
This is a main difference from our earlier work \cite{GPM3}, where Lipschitz regularity is established for the solutions of the Hamilton-Jacobi equation.

The main result of the paper is:
\begin{teo}\label{mainresult}
Assume that the Assumptions A\ref{ham}-A\ref{alpha}, from Section \ref{ma} hold.
Then, for every $R>0$ there exists a constant $C_R>0$, such that 
any solution  $(u,m)$ to \eqref{mfg}, satisfies
$$\left\|Du\right\|_{L^\infty(0,T;L^p(\mathcal{B}_R))}\leq C_R.$$
\end{teo}

After this a-priori bound is derived, one can prove, using standard methods, additional regularity for the solutions. This will not be pursued here as it would follow the same steps as in our previous results, see \cite{GPM2}, \cite{GPM3}. To illustrate this point, we just give an example of how this could be further developed.

Notice that $m\in L^\infty(0,T;L^1(\Tt^d))$. Also, because of Lemma \ref{lemma2}, we have $m\in L^{\af+1}(0,T;L^\frac{2^*(\af+1)}{2}(\Tt^d))$, where $$2^*=\frac{2d}{d-2}.$$ Then, estimates for $m$ in several Lebesgue spaces can be obtained by interpolation. Moreover, by multiplying the second equation in \eqref{mfg} by $\phi^2m^\beta$, where $\phi$ is a spatial cut-off function, one obtains, by standard techniques, bounds for $$\left\|\phi m\right\|_{L^\infty(0,T;L^{\beta+1}(\Tt^d))},$$ and $$\left\|D\left(\phi m^\frac{\beta+1}{2}\right)\right\|_{L^2(0,T;L^{2}(\Tt^d))}.$$ This can be iterated as in \cite{GPM2}, for example, to produce improved integrability for $m$. 

Section \ref{assumptions} introduces the technical assumptions under which we work. An outline of the proof of Theorem \ref{mainresult} is presented in Section \ref{outline}. A number of auxiliary results are presented in Section \ref{be}. The proof of Theorem \ref{mainresult} is given in Section \ref{secsob} by developing the adjoint method in Sobolev spaces.

\section{Main assumptions and outline of the proof}\label{ma}

In this section, we present the set of assumptions under which we will work, as well as an outline of the proof of Theorem \ref{mainresult}. 

\subsection{Main assumptions}\label{assumptions}

\begin{Assumption}\label{ham}
We assume that the Hamiltonian
$H\in\mathcal{C}^\infty(\Rr^d\times \Rr^d)$ is strictly convex in the second variable. Furthermore, we suppose that $H$ is non-negative and coercive, i.e.,
\[
\lim_{|p|\to \infty} \frac{H(x,p)}{|p|}=+\infty.
\] 
\end{Assumption}

\begin{Assumption}\label{ic}
The functions $u_T$ and $m_0$ are  smooth and integrable. Furthermore, $m_0$ is non-negative, compactly supported, and satisfies $$\int_{\Rr^d}m_0dx\;=\;1.$$
\end{Assumption}

The Lagrangian $L(x,v)$ is defined as follows:
$$L(x,v)\doteq \sup_p-p\cdot v-H(x,p).$$

\begin{Assumption}\label{lagrange}
The Lagrangian $L(x,v):\Rr^d\times\Rr^d\mapsto\Rr$ is such that $L(x,0)\in L^1(\Rr^d)\cap L^\infty(\Rr^d)$ and $L(x,0)\geq 0$.
\end{Assumption}

\begin{Assumption}\label{dphminush}For some constants $c,\,C>0$
$$D_pH(x,p)\cdot p-H(x,p)\geq cH(x,p)-C.$$
\end{Assumption}

\begin{Assumption}\label{ge}
The non-linearity $g$ is increasing. Also, there exists $C>0$, such that
\begin{equation}
g[m](x,t)\leq
\begin{cases}
C m(x,t),&\;  m\leq 1\\
C m^\alpha(x,t),&\; m> 1,
\end{cases}
\end{equation} with $g(0)=0$.
Since $g$ is increasing, it is the derivative of a convex function $G:\Rr\to\Rr$. We assume that $G$ is such that, 
for $z>1$,
 $$C_1z^{\alpha+1}\leq G(z)\leq C_2z^{\alpha+1}.$$
\end{Assumption}

\begin{Assumption}
\label{asbdpph}
The Hamiltonian $H$ is such that
$\left|D^2_{xp}H\right|^2\leq CH$
and, for any symmetric matrix $M$
\[
\left|D^2_{pp}HM\right|^2\leq  C \tr(D^2_{pp} H M M). 
\]
\end{Assumption}


\begin{Assumption}\label{dphsq}There is $C>0$ such that
$$\left|D_pH(x,p)\right|^2\leq C+ CH(x,p).$$
\end{Assumption}

\begin{Assumption}\label{dxh}There exists $C>0$ such that
$$\left|D_xH(x,p)\right|\leq C+ CH(x,p).$$
\end{Assumption}

\begin{Assumption}\label{alpha}The exponent $\alpha$ is such that
\begin{equation}
0\,<\,\alpha\,<\,
\frac{1}{d-1}
\end{equation}
\end{Assumption}

The exponent $\alpha$ in the previous assumption is critical in the following arguments, see Lemma \ref{lt4}.


\subsection{Outline of the proof}\label{outline}

In order to justify rigorously our computations, one needs to consider a regularized version of \eqref{mfg}.
In this, the local non-linearity $g$ is replaced by the non-local operator $$g_\epsilon(m)\doteq\eta_\ep\ast g\left(\eta_\ep\ast m\right),$$
where  $\eta_\ep$ is a symmetric standard mollifying kernel.
We assume that $g_0=g$. This procedure yields the following regularized system:

\begin{equation}\label{smfg}
\begin{cases}
-u^\epsilon_t+H(x,Du^\epsilon)=\Delta u^\epsilon+g_\epsilon(m^\epsilon)\\m^\epsilon_t-\div(D_pH m^\epsilon)=\Delta m^\epsilon.
\end{cases}
\end{equation}



The existence of smooth solutions to \eqref{smfg},
can be proved adapting the ideas in \cite{cardaliaguet} (see also \cite{GPM2}). However, for the sake of simplicity, we consider the original system throughout the paper, and establish a-priori estimates.

To obtain Theorem \ref{mainresult}, we first investigate the Sobolev regularity of the solutions of the Hamilton-Jacobi, as stated in the following Proposition. 

\begin{Proposition}\label{proposition1}
Let $(u,m)$ be a solution to \eqref{mfg} and suppose that Assumptions A\ref{ham}-\ref{dxh} are satisfied.
Suppose $a$ and $b$ satisfy \eqref{te1}. 
Then
there exist $C_R, \theta_1, \theta_2$ such that 
$$\left\|Du\right\|_{L^\infty(0,T;L^p(\mathcal{B}_R))}\leq C_R+C_R\left\|g\right\|_{L^c(0,T;L^a(\Rr^d))}^{\theta_1}+C_R\left\|g\right\|_{L^c(0,T;L^a(\Rr^d))}^{\theta_2}.$$
\end{Proposition}
This is accomplished by combining the non-linear adjoint method with an improved integrability estimate for the adjoint variable. These new ideas are crucial to circumvent the lack of integrability of $H$.
Proposition \ref{proposition1} is proven in Section \ref{sobolev}. To prove $W^{1,p}_{Loc}(\Rr^d)$-regularity for $u$, it is critical to control the integrability of $g$ with respect to both time and space. This is done in the next Lemma:

\begin{Lemma}\label{lemma2}
Let $(u,m)$ be a solution to \eqref{mfg} and suppose that assumptions A\ref{ham}-A\ref{asbdpph} are satisfied. Then, there exists a constant $C>0$ such that $$\left\|m\right\|_{L^{\af+1}(0,T;L^{\frac{2^*(\af+1)}{2}}(\Rr^d))}\,\leq\,C.$$
\end{Lemma} The proof of Lemma \ref{lemma2} is given in Section \ref{be}.

Because $g$ is a power-like non-linearity, the estimate in Lemma \ref{lemma2} yields an upper bound for norms of $g$ in some appropriate Lebesgue space.
This Lemma is then combined with  Proposition \ref{proposition1} and the technical Lemma \ref{lt4}
to prove 
Theorem \ref{mainresult}.

\section{Basic estimates}\label{be}

In this Section, we obtain various estimates for solutions of \eqref{smfg}. These bounds are similar to the ones
 for the periodic setting, in \cite{ll1,ll2,GPM2}, and the arguments to prove them are not substantially modified for $\Rr^d$. Consequently, they are only discussed here briefly, for convenience.

We begin by considering the auxiliary equation:
\begin{equation}
\begin{cases}
\zeta_t+\div(b\zeta)=\Delta\zeta\\\zeta(x,\tau)=\zeta_0(x),
\end{cases}
\end{equation}where $b:\Rr^d\times(\tau,T)\to\Rr^d$ is a smooth vector field, $0<\tau<T$ is arbitrary and $\zeta_0$ is a given initial condition.

\begin{Lemma}\label{lemma1}
Let $(u,m)$ be a solution to \eqref{mfg} and assume that Assumptions A\ref{ham}-A\ref{lagrange} hold. Then,
\begin{enumerate}
\item
\begin{align}
\label{etreze}
\int_{\Rr^d}u(x,\tau)m_0 dx&\le
CT+\int_\tau^T\int_{\Rr^d}g(m)(x,t)\zeta^{m_0}(x,t) dxdt \\\nonumber&
\quad+\int_{\Rr^d} u(x,T) \zeta^{m_0}(x,T)dx,
\end{align}
where $\zeta^{m_0}(x,t)$ is the solution to the heat equation with $\zeta_0=m_0$. 
Also,
\item 
\begin{equation}
\label{ecatorze}
\int_{\mathcal{B}_R} u(x,\tau)dx\leq CT +\int_\tau^T \int_{\Rr^d}
g(m)(x,t)\zeta^{\chi_{\mathcal{B}_R}}(x,t)dxdt, 
\end{equation}where $\zeta^{\chi_{\mathcal{B}_R}}(x,t)$ is the solution to
the heat equation with initial condition $\zeta_0=\chi_{\mathcal{B}_R}$, for arbitrarily fixed $R>0$, and
$\chi_{\mathcal{B}_R}$ denotes the characteristic function of $\mathcal{B}_R$.
\end{enumerate}
\end{Lemma}
\begin{proof}
We have as in \cite{GPM2}
\begin{align*}
\label{lhe2}
\int_{\Tt^d} u(x,t)\zeta_0(x) dx \le
&\int_t^T\int_{\Tt^d}\bigl(L(y,b(y,s))
+g(m)(y,s)\bigr)\zeta(y,s) dy ds\\\nonumber
&+\int_{\Tt^d} u(y, T) \zeta(y,T). 
\end{align*}Set $b\,\equiv\,0$. Consider first the case $\zeta_0\,\equiv\,m_0$. Because $\zeta^{m_0}$ is a probability measure and $L(x,0)$ is bounded,  one obtains $$\int_{\Rr^d}L(x,0)\zeta^{m_0}(x,t)dx\leq C,$$for some $C>0$. This implies \eqref{etreze}. To establish \eqref{ecatorze} we set $b\equiv 0$ and $\zeta_0=\chi_{\mathcal{B}_R}$. 
\end{proof}

Next, we recover the first-order estimates in the whole space setting. 

\begin{Proposition}
\label{pehm}
Assume A\ref{ham}-A\ref{ge} hold. 
Let $(u, m)$ be a solution of \eqref{mfg}. Then
\begin{equation*}
\label{ihm}
\int_0^T \int_{\Tt^d} c H(x, D_xu) m
+G(m) dx dt 
\leq CT+C \left\|u(\cdot, T)\right\|_{L^\infty(\Tt^d)}, 
\end{equation*}
where $G'=g$.
\end{Proposition}

\begin{proof}
We have
\[
-\frac{d}{dt}\int_{\Tt^d}u^\epsilon m^\epsilon dx+\int_{\Tt^d}(H-D_pH
D_xu^\epsilon)m^\epsilon dx =\int_{\Tt^d}m^\epsilon g_\epsilon(m^\epsilon)dx.
\]
Assumption A\ref{dphminush}, leads to
\begin{align*}
&c \int_0^T\int_{\Tt^d} H(x,Du) m dx\,dt\leq 
\int_0^T\int_{\Tt^d} (D_pH D_xu-H) m dx\,dt=\\
&-\int_0^T\int_{\Tt^d}mg(m)dx+\int_{\Tt^d}\left( u(x,0)
  m(x, 0)- u(x,T) m(x, T)\right)dx.  
\end{align*}
By using the first assertion in Lemma \ref{lemma1} one obtains
\begin{align*}
c \int_0^T\int_{\Tt^d}H(x,Du)mdx\,dt \leq 
&CT+\int_{\Tt^d} u(x,T)(\zeta^{m_0}(x,T)- m(x,T))dx\\
&+\int_0^T\int_{\Tt^d} g(m)(\zeta^{m_0} -m) dx\,dt,
\end{align*}where $\zeta$ solves the heat equation with $\zeta_0(x)=m_0(x)$.

Assumptions A\ref{ge} ensures the existence of a convex function $G$ with
$G'(z)=g(z)$. Hence, 
$
g(m)(\zeta^{m_0}-m) \leq
G(\zeta^{m_0}) -G(m), 
$
and then,
\begin{align*}
c \int_0^T\int_{\Tt^d} H(x, Du) m dx\,dt 
&+\int_0^T\int_{\Tt^d} G(m) dxdt\\ 
\qquad
&\leq CT +\left\|u(\cdot , T)\right\|_{L^\infty(\Tt^d)}+\int_0^T\int_{\Tt^d}
G(\zeta^{m_0}) dxdt.
\end{align*}
It remains for us to control $$\int_0^T\int_{\Tt^d}
G(\zeta^{m_0}) dxdt.$$ Because of A\ref{ge} we have
$$\int_0^T\int_{\Tt^d}
G(\zeta^{m_0}) dxdt\leq C\int_0^T\int_{\Tt^d}\left(\zeta^{m_0}\right)^{\alpha+1}\leq C\int_0^T\int_{\Tt^d}\left(\theta*m_0\right)^{\alpha+1},$$where $\theta$ is the heat kernel. Therefore, A\ref{ic} together with the Young's inequality for convolutions implies the result.
\end{proof}

\begin{Corollary}\label{corollary1}
Assume A\ref{ham}-A\ref{ge} hold. 
Let $(u,m)$ be a solution of \eqref{mfg}. Then 
\[
\int_0^T\int_{\Tt^d}m^{\alpha+1} 
+H(x,Du) m dxdt \leq C. 
\]
\end{Corollary}

In what follows, we recover the second-order estimate in the whole space. Since its proof 
is similar to the one in
 the periodic setting, it is omitted here (we refer the reader to \cite{GPM2}).

\begin{Proposition}\label{soe}
Assume A\ref{ham}-\ref{asbdpph} hold. 
Let $(u,m)$ be a solution of \eqref{mfg}. Then
\begin{align*}
&\int_0^T\int_{\Tt^d}  g'(m)|D_x m|^2
+
\tr(D_{pp}^2H(D^2_{xx}u)^2) m
\leq \max_x\Delta u(x,T)\\&+
 C(1+\max_x u(x,T) -\min_xu(x, T))-\int_{\Tt^d} u(x,0)\Delta m_0(x) dx.
\end{align*}
\end{Proposition}


Finally, we have all the ingredients needed for the proof of Lemma \ref{lemma2}:

\begin{proof}[Proof of Lemma \ref{lemma2}.]
We have
\begin{align*}
\int_0^T \|m\|_{L^{\frac{2^*}2 (\alpha+1)}(\Tt^d)}^{\alpha +1}dt&\leq C\int_{0}^T\int_{\Rr^d}\left(\eta_\ep\ast m^\ep\right)^{\af+1}dxdt\\&\quad+\int_0^T\int_{\Tt^d}g'(m)|D_xm|^2dx dt.
\end{align*}The result follows from Corollary \ref{corollary1} together with Proposition \ref{soe}.
\end{proof}

\section{Sobolev regularity for the Hamilton-Jacobi equation}\label{secsob}

In this Section, we investigate the local Sobolev regularity of $Du$. Unlike
in
the periodic case, where we prove that a-priori $Du\in L^\infty(\Tt^d)$, here we obtain
a weaker bound, namely
$Du\in L^p_{loc}(\Rr^d)$. We start by establishing some preliminary estimates.

\subsection{Preliminary estimates}

We consider the adjoint equation:
\begin{equation}\label{adjlnurdws}
\begin{cases}
\zeta_t-\div\left(D_pH\zeta\right)=\Delta\zeta\\\zeta(x,\tau)=\phi(x),
\end{cases}
\end{equation}where $0\leq \tau<T$ and $\phi\geq 0$ is such that $\phi\in L^1(\Rr^d)\cap L^{p'}(\Rr^d)$ with $\left\|\phi\right\|_{L^{p'}(\Rr^d)}=1$. Assume further that $\phi$ has compact support. Note that, there exists a constant $C>0$, which depends on $p'$ and the support of $\phi$, for which $$\|\phi\|_{L^1(\Rr^d)}=\int_{\Rr^d}\phi(x)dx\leq C.$$
Because $$\frac{d}{dt}\int_{\Rr^d}\zeta(x,t)dx=0,$$we have 
\begin{equation}\label{estrela}
\|\zeta\|_{L^1(\Rr^d)}=\int_{\Rr^d}\zeta(x,t)dx\leq C,
\end{equation}for $t\in\left[0,T\right]$. 


\begin{Lemma}\label{lt1}
Let $d>2$ and assume that $a,\,a',\,c,\,c'\,>1$ satisfy 
\begin{equation}\label{eq1}
\frac{1}{a}+\frac{1}{a'}=1,
\end{equation}and
\begin{equation}\label{eq4}
\frac{1}{c}+\frac{1}{c'}=1.
\end{equation}
Then, there exist $p,\,p',\,q,\,q'>1$ so that 
\begin{equation}\label{eq2}
\frac{1}{p'}+\frac{1}{q}=\frac{1}{a'}+1,
\end{equation}
\begin{equation}\label{eq3}
\frac{dc'}{2q'}<1,
\end{equation}
\begin{equation}\label{eq4'}
\frac{1}{p}+\frac{1}{p'}=1,
\end{equation}
and
\begin{equation}\label{eq5}
\frac{1}{q}+\frac{1}{q'}=1
\end{equation} hold simultaneously.
\end{Lemma}
\begin{proof}
The Lemma follows from elementary computations and can be verified by using the software \textit{Mathematica}.
\end{proof}

\begin{Lemma}\label{lemma1prelim}
Let $(u,m)$ be a solution to \eqref{mfg} and suppose that Assumptions A\ref{ham}-A\ref{dphminush} hold. Additionally, let $a,\,a',\,c,\,c'\,>\,1$ satisfying \eqref{eq1} and \eqref{eq4}. Then, 
$$\left|\int_{\Rr^d}u(x,\tau)\phi(x)dx\right|\leq C+C\|g\|_{L^c(0,T;L^a(\Rr^d))}\left(1+\|\zeta\|_{L^{c'}(0,T;L^{a'}(\Rr^d))}\right),$$where $\zeta$ solves \eqref{adjlnurdws}.
\end{Lemma}
\begin{proof}Because of Lemma \ref{lt1} we can fix $p,\,p',\,q,\,q'>1$ such that \eqref{eq2}-\eqref{eq5} hold. Therefore we have, for $\rho=\phi\ast\theta$,
\begin{align*}
\int_{\Rr^d}u(x,\tau)\phi(x)dx&\leq\int_\tau^T\int_{\Rr^d}\left(L(x,0)+g(m)\right)\rho dxdt+\int_{\Rr^d}u_T(x)\rho(x,T)\\&\leq CT+\int_0^T\int_{\Rr^d}g(m)\left(\phi\ast\theta\right)dxdt,
\end{align*}where $\theta$ is the heat kernel.
H\"older's inequality implies, for $a>1$ and $a'$ given by \eqref{eq1},
$$\int_0^T\int_{\Rr^d}g(m)(\phi\ast\theta)dxdt\leq \int_0^T\|g\|_{L^a(\Rr^d)}\|\phi\ast\theta\|_{L^{a'}(\Rr^d)}.$$ Because of \eqref{eq2}, Young's inequality for convolution leads to
\begin{align*}
\int_0^T\|g\|_{L^a(\Rr^d)}\|\phi\ast\theta\|_{L^{a'}(\Rr^d)}dt&\leq \int_0^T\|g\|_{L^a(\Rr^d)}\|\phi\|_{L^{p'}(\Rr^d)}\|\theta\|_{L^{q}(\Rr^d)}dt\\&\leq\int_0^T\|g\|_{L^a(\Rr^d)} \frac{C\|\phi\|_{L^{p'}(\Rr^d)}}{t^\frac{d}{2q'}}dt\\&\leq C\|g\|_{L^c(0,T;L^a(\Rr^d))},
\end{align*}
using \eqref{eq3}.
By gathering the previous computation, it follows that 
$$\int_{\Rr^d}u(x,\tau)\phi(x)dx\leq CT+C\|g\|_{L^c(0,T;L^a(\Rr^d))}.$$

On the other hand, let $\zeta$ be a solution to \eqref{adjlnurdws}. Then,
\begin{align*}
\int_{\Rr^d}u(x,\tau)\phi(x)dx&=\int_\tau^T\int_{\Rr^d}\left(D_pHDu-H+g\right)\zeta dxdt+\int_{\Rr^d}u_T(x)\zeta(x,T)\\&\geq -CT-\int_0^T\int_{\Rr^d}|g\zeta|dxdt\\&\geq -CT-\|g\|_{L^c(0,T;L^a(\Rr^d))}\|\zeta\|_{L^{c'}(0,T;L^{a'}(\Rr^d))},
\end{align*}which yields the result.
\end{proof}

\begin{Lemma}\label{lemmahzeta}
Let $(u,m)$ be a solution to \eqref{mfg} and suppose that Assumptions A\ref{ham}-A\ref{dphminush} are satisfied. Let $\zeta$ solve  \eqref{adjlnurdws}. Let $a,\,a',\,c,\,c'>1$
satisfying \eqref{eq1} and \eqref{eq4}. Then,
$$\int_0^T\int_{\Rr^d}H(x,Du)\zeta(x,t)dxdt\leq C+C\|g\|_{L^c(0,T;L^a(\Rr^d))}\left(1+\|\zeta\|_{L^{c'}(0,T;L^{a'}(\Rr^d))}\right).$$
\end{Lemma}
\begin{proof}
Observe that
\begin{align*}
\int_{\Rr^d}u(x,0)\phi(x)dx&=\int_\tau^T\int_{\Rr^d}\left(D_pHDu-H+g\right)\zeta dxdt+\int_{\Rr^d}u_T(x)\zeta(x,T)\\&\geq -C+C\int_0^T\int_{\Rr^d}H(x,Du)\zeta dxdt+\int_0^T\int_{\Rr^d}g\zeta dxdt.
\end{align*}The result follows then from Lemma \ref{lemma1prelim}.
\end{proof}

\subsection{Entropy dissipation}\label{sobolev}

We start with an auxiliary lemma.

\begin{Lemma}\label{lemma3}
Let $\zeta$ be a solution to \eqref{adjlnurdws}. Then, 
$$
\int_{\Rr^d}\left(1+|x|^2\right)^\frac{1}{2}\zeta(x,r)dx\,\leq\,C\,+\,C\int_\tau^r\int_{\Rr^d}|D_pH|^2\zeta dxdt,
$$for $\tau\in\left[0,T\right)$ and $r\in\left(\tau,T\right]$.
\end{Lemma}
\begin{proof}
Notice that 
\begin{align}\label{1}
\frac{d}{dt}\int_{\Rr^d}\left(1+|x|^2\right)^\frac{1}{2}\zeta\,&=\,\int_{\Rr^d}\left(1\,+\,|x|^2\right)^\frac{1}{2}\div(D_pH\zeta)dx\\\nonumber&\quad+\int_{\Rr^d}\left(1+|x|^2\right)^\frac{1}{2}\Delta \zeta dx.
\end{align}Observe that
\begin{align*}
\int_{\Rr^d}\left(1+|x|^2\right)^\frac{1}{2}\Delta \zeta dx\,&=\,-\int_{\Rr^d}\frac{|x|^2\zeta}{\left(1+|x|^2\right)^\frac{3}{2}} dx+\int_{\Rr^d}\frac{d\zeta}{\left(1+|x|^2\right)^\frac{1}{2}}dx\leq C,
\end{align*}using \eqref{estrela}. It remains for us to address the first term in the right-hand side of \eqref{1}. We have
\begin{align*}
\int_{\Rr^d}\left(1\,+\,|x|^2\right)^\frac{1}{2}\div(D_pH\zeta)dx&\leq C\int_{\Rr^d}\left(1+|x|^2\right)^{-1}|x|^2\zeta dx+\int_{\Rr^d}|D_pH|^2\zeta dx.
\end{align*}
Notice that
\begin{align*}
\int_{\Rr^d}\left(1+|x|^2\right)^{-1}|x|^2\zeta dx&\leq C.
\end{align*}Hence,
\begin{align*}
\frac{d}{dt}\int_{\Rr^d}\left(1+|x|^2\right)^\frac{1}{2}\zeta&\leq C+C\int_{\Rr^d}|D_pH|^2\zeta dx.
\end{align*}By integrating the former inequality in time over $\left(\tau,r\right)$ one obtains 
\begin{align*}
\int_{\Rr^d}\left(1+|x|^2\right)^\frac{1}{2}\zeta(x,r)dx&\leq C+C\int_\tau^r\int_{\Rr^d}|D_pH|^2\zeta dx+\int_{\Rr^d}\left(1+|x|^2\right)^\frac{1}{2}\phi(x) dx,
\end{align*}which proves the result, since $\phi$ has compact support.
\end{proof}

Using the previous Lemma we obtain the following entropy dissipation estimate:

\begin{Lemma}\label{lem1rd}
Let $\zeta$ solve \eqref{adjlnurdws}. Then, there exists $C>0$ for which $$\int_{\Rr^d}\zeta(x,r) \ln\left[\zeta(x,r)\right]dx\geq -C-C\int_\tau^r\int_{\Rr^d}|D_pH|^2\zeta dx,$$for $\tau\in\left[0,T\right)$ and $r\in\left(\tau,T\right]$.
\end{Lemma}
\begin{proof}
Let $C_{d,p}$ be such that $$\int_{\Rr^d}\left(1+|x|^2\right)^{-p}dx=\frac{1}{C_{d,p}}.$$Notice that 
\begin{align}\label{zlnz}
C_{d,p}\int_{\Rr^d}\zeta(x,t)\ln\left[\zeta(x,t)\right]dx&=C_{d,p}\int_{\Rr^d}\left(1+|x|^2\right)^p\left(1+|x|^2\right)^{-p}\zeta\ln\left[\left(1+|x|^2\right)^p\zeta\right]\\\nonumber&\quad-C_{d,p}\int_{\Rr^d}\left(1+|x|^2\right)^p\left(1+|x|^2\right)^{-p}\zeta\ln\left[\left(1+|x|^2\right)^p\right].
\end{align}The first term in the right-hand side of \eqref{zlnz} can be written as $$\int_{\Rr^d}\Psi\left[\psi(x)\right]d\mu(x),$$where $$\mu(x)=C_{d,p}\left(1+|x|^2\right)^{-p},$$ $$\psi(x)=\left(1+|x|^2\right)^{p}\zeta,$$and$$\Psi(y)=y\ln(y).$$ Because $\Psi$ is a convex function, Jensen's inequality yields
\begin{align*}\label{zlnz1}
&C_{d,p}\int_{\Rr^d}\left(1+|x|^2\right)^p\left(1+|x|^2\right)^{-p}\zeta\ln\left[\left(1+|x|^2\right)^p\zeta\right]dx \\&\geq C_{d,p}\left[ \int_{\Rr^d}\left(1+|x|^2\right)^p\left(1+|x|^2\right)^{-p}\zeta dx\right]\ln\left[C_{d,p}\int_{\Rr^d}\left(1+|x|^2\right)^p\zeta\left(1+|x|^2\right)^{-p}\right]\\&\geq -C,
\end{align*}for some $C>0$.

On the other hand, the second term in the right-hand side of \eqref{zlnz} is
\begin{align*}
-C_{d,p}\int_{\Rr^d}\left(1+|x|^2\right)^p\left(1+|x|^2\right)^{-p}\zeta\ln\left[\left(1+|x|^2\right)^p\right]&=-C_{d,p}\int_{\Rr^d}\ln\left(\left[1+|x|^2\right]^p\right)d\zeta(x)\\&\geq 
-C\int_{\Rr^d}\left(1+|x|^2\right)^\frac{1}{2}\zeta dx,
\end{align*}where the inequality follows from the fact that $$\ln\left(\left(1+|x|^2\right)^p\right)\leq 2p\left(1+|x|^2\right)^\frac{1}{2}.$$We observe that the last inequality follows from Jensen's, since $\zeta$ is a probability measure, combined with the sublinearity of the logarithmic function. Therefore, Lemma \ref{lemma3} implies
\begin{align*}
-C_{d,p}\int_{\Rr^d}\left(1+|x|^2\right)^p&\left(1+|x|^2\right)^{-p}\zeta\ln\left[\left(1+|x|^2\right)^p\right]\\&\geq -C-C\int_0^T\int_{\Rr^d}|D_pH|^2\zeta dxdt.
\end{align*}
\end{proof}

\subsection{Sobolev regularity}

\begin{Proposition}\label{drhohalf}
Let $\zeta$ be a solution to \eqref{adjlnurdws}. Then, there exists $C>0$ such that
$$\int_\tau^T\int_{\Rr^d}\left|D\zeta^\frac{1}{2}(x,t)\right|^2dxdt\leq C\,+\,C\int_\tau^T\int_{\Rr^d}|D_pH|^2\zeta dxdt.$$
\end{Proposition}
\begin{proof}
Multiply \eqref{adjlnurdws} by $\ln\left[\zeta(x,t)\right]$ and integrate by parts to obtain
\begin{align*}
\frac{d}{dt}\int_{\Rr^d}\zeta(x,t)\ln\left[\zeta(x,t)\right]dx&=-\int_{\Rr^d}D_pH\zeta^\frac{1}{2}\zeta^{-\frac{1}{2}}D\zeta dx - 4\int_{\Rr^d}\left|D\zeta^\frac{1}{2}\right|^2dx\\&\leq C\int_{\Rr^d}\left|D_pH\right|^2\zeta(x,t)dx\,-\,C\int_{\Rr^d}\left|D\zeta^\frac{1}{2}\right|^2dx.
\end{align*}
Integrating in time on $\left(\tau,T\right)$, it follows 
\begin{align*}
\int_\tau^T\int_{\Rr^d}\left|D\zeta^\frac{1}{2}\right|^2dxdt&\leq C\int_\tau^T\int_{\Rr^d}\left|D_pH\right|^2\zeta dxdt+\int_{\Rr^d}\zeta(x,\tau)\ln\left[\zeta(x,\tau)\right]dx\\&\quad-\int_{\Rr^d}\zeta(x,T)\ln\left[\zeta(x,T)\right]dx\\&\leq C\,+\,C\int_\tau^T\int_{\Rr^d}\left|D_pH\right|^2\zeta dxdt,
\end{align*}
using Lemma \ref{lem1rd} in the last inequality.
\end{proof}

\begin{Corollary}\label{hzeta}
Let $(u,m)$ solve \eqref{mfg} and assume that Assumptions A\ref{ham}-A\ref{ge} hold. Suppose that
$a, a',c, c'>1$ satisfy
\eqref{eq1} and \eqref{eq4}. Then, 
$$\int_{\tau}^T\int_{\Rr^d}\left|D\zeta^\frac{1}{2}\right|^2dxdt\leq C+C\|g\|_{L^c(0,T;L^a(\Rr^d))}\left(1+\|\zeta\|_{L^{c'}(0,T;L^{a'}(\Rr^d))}\right).$$
\end{Corollary}
\begin{proof}
The result follows from A\ref{dphsq} combined with Lemma \ref{lemmahzeta} and Proposition \ref{drhohalf}.
\end{proof}

Next, we control norms of $Du$ in $L^\infty(0,T;L^{p'}(\Rr^d))$.

\begin{Lemma}\label{lt2}
Let $d>2$ and assume that $a$ and $c$ satisfy
\begin{equation}\label{te1}
a\geq \frac{c d}{c-2}\,\,\,\mbox{and}\,\,\,c>2.
\end{equation}Then, there exist $\tilde{s},\,b\,>\,1$ and $0<\lambda<1$ such that 
\begin{equation}\label{eq6}
\frac{1}{c}+\frac{1}{\tilde{s}}=\frac{1}{2},
\end{equation} 
\begin{equation}\label{eq7}
\frac{1}{a}+\frac{1}{b}=\frac{1}{2},
\end{equation}
\begin{equation}\label{eqlambda}
\frac{2}{b}=1-\lambda+\frac{2\lambda}{2^*},
\end{equation}
and
\begin{equation}\label{eq8}
\frac{\tilde{s}\lambda}{2}\leq 1,
\end{equation}are mutually satisfied.
\end{Lemma}
{\bf Remark:}
Note that if $a$ and $c$ satisfy \eqref{te1} then $a>1$ and $c>1$, therefore the requirements of Lemma
\ref{lt1} hold. 

\begin{proof}
As before, the result is established by various elementary computations and can be checked by using the software \textit{Mathematica}.
\end{proof}

\begin{Proposition}\label{dulprd}
Let $(u,m)$ be a solution to \eqref{mfg} and assume that Assumptions A\ref{ham}-A\ref{dxh} are satisfied. Let $\zeta$ solve \eqref{adjlnurdws} and suppose that $a$ and $c$ satisfy \eqref{te1}. Let $a'$ and $c'$ be 
given by \eqref{eq1} and \eqref{eq4}. 
Then,
\begin{align*}
\left|\int_{\Rr^d}u_\xi(x,t)\phi(x)dx\right|&\leq C+C\left\|g\right\|_{L^c(0,T;L^{a}(\Rr^d))}\left(1+\left\|\zeta\right\|_{L^{c'}(0,T;L^{a'}(\Rr^d))}\right)\\
&\quad+C\left\|g\right\|_{L^c(0,T;L^{a}(\Rr^d))}^\frac{3\tilde{s}+2}{2\tilde{s}}\left(1+\left\|\zeta\right\|_{L^{c'}(0,T;L^{a'}(\Rr^d))}^{\frac{2+\tilde{s}}{2\tilde{s}}}\right),
\end{align*}where $u_\xi$ is the derivative of $u$ with respect to the spatial direction $\xi$.
\end{Proposition}
\begin{proof}
We start by fixing a unit vector $\xi$. We differentiate \eqref{mfg} in the $\xi$ direction, multiply it by $\zeta$ and \eqref{adjlnurdws} by $u_\xi$. By adding the
resulting identities, and integrating by parts, one obtains
$$\left|\int_{\Rr^d}u_\xi(x,\tau)\phi(x)dx\right|=\left|\int_\tau^T\int_{\Rr^d}D_\xi H\zeta+g_\xi\zeta dxdt+\int_{\Rr^d}u_{\zeta}(x,T)\zeta(x,T)dx.\right|$$ Because of A\ref{dxh} and Lemma \ref{lemmahzeta}, it follows that
\begin{align*}
\left|\int_{\Rr^d}u_\xi(x,t)\phi(x)\right|&\leq C+C\left\|g\right\|_{L^c(0,T;L^{a}(\Rr^d))}\left(1+\left\|\zeta\right\|_{L^{c'}(0,T;L^{a'}(\Rr^d))}\right)\\\nonumber&\quad+\left|\int_\tau^T\int_{\Rr^d}g_\xi\zeta dxdt\right|.
\end{align*}
It remains to address the term $$\left|\int_\tau^T\int_{\Rr^d}g_\xi\zeta dxdt\right|.$$ 

Before we proceed, choose $b$, $\tilde s$ and $\lambda$ as in Lemma \ref{lt2} so that conditions \eqref{eq6}-\eqref{eq8} are mutually satisfied. As a consequence, we have
$$\left|\int_\tau^T\int_{\Rr^d}g_\xi\zeta dxdt\right|\leq C\left\|g\right\|_{L^c(0,T;L^{a}(\Rr^d))}\left\|\zeta^\frac{1}{2}\right\|_{L^{\tilde{s}}(0,T;L^{b}(\Rr^d))}\left\|D\zeta^\frac{1}{2}\right\|_{L^2(0,T;L^{2}(\Rr^d))}.$$
Corollary \ref{hzeta} controls $\left\|D\zeta^\frac{1}{2}\right\|_{L^2(0,T;L^{2}(\Rr^d))}$ in terms of norms of $g$ and $\zeta$. We investigate next upper bounds for $$\left\|\zeta^\frac{1}{2}\right\|_{L^{\tilde{s}}(0,T;L^{b}(\Rr^d))}=\left[\int_\tau^T\left(\int_{\Rr^d}\zeta^\frac{b}{2}\right)^\frac{\tilde{s}}{b}\right]^\frac{1}{\tilde{s}}.$$
H\"older's inequality yields
$$\left(\int_{\Rr^d}\zeta^\frac{b}{2}\right)^\frac{2}{b}\leq \left(\int_{\Rr^d}\zeta\right)^{1-\lambda}\left(\int_{\Rr^d}\zeta^{\frac{2^*}{2}}\right)^\frac{2\lambda}{2^*},$$
once condition \eqref{eqlambda} holds.

We proceed by addressing $$\left(\int_{\Rr^d}\zeta^{\frac{2^*}{2}}\right)^\frac{2\lambda}{2^*}.$$ Since $\zeta\in L^1(\Rr^d)$, Gagliardo-Nirenberg inequality yields $$\left(\int_{\Rr^d}\zeta^\frac{2^*}{2}dx\right)^\frac{2\lambda}{2^*}\leq C+C\left(\int_{\Rr^d}|D\zeta^\frac{1}{2}|^2dx\right)^\lambda.$$ Therefore, because of \eqref{eq8}, 
we have
$$\left\|\zeta^\frac{1}{2}\right\|_{L^{\tilde{s}}(0,T;L^b(\Rr^d))}\leq C+C\left\|D\zeta^\frac{1}{2}\right\|^\frac{2}{\tilde{s}}_{L^2(\Rr^d\times\left[0,T\right])}.$$ Hence,
\begin{align*}
\left|\int_\tau^T\int_{\Rr^d}g_\xi\zeta dxdt\right|&\leq C\left\|g\right\|_{L^c(0,T;L^{a}(\Rr^d))}\left(C+\left\|D\zeta^\frac{1}{2}\right\|_{L^2(0,T;L^{2}(\Rr^d))}^\frac{2}{\tilde{s}}\right)\left\|D\zeta^\frac{1}{2}\right\|_{L^2(0,T;L^{2}(\Rr^d))}\\&\leq C\left\|g\right\|_{L^c(0,T;L^{a}(\Rr^d))}\left(C+\left\|D\zeta^\frac{1}{2}\right\|_{L^2(0,T;L^{2}(\Rr^d))}^\frac{2+\tilde{s}}{\tilde{s}}\right)\\&\leq C+C\left\|g\right\|_{L^c(0,T;L^{a}(\Rr^d))}^\frac{3\tilde{s}+2}{2\tilde{s}}\left(1+\left\|\zeta\right\|_{L^{c'}(0,T;L^{a'}(\Rr^d))}^\frac{2+\tilde{s}}{2\tilde{s}}\right).
\end{align*}Lastly, 
\begin{align*}
\left|\int_{\Rr^d}u_\xi(x,t)\phi(x)dx\right|&\leq C+C\left\|g\right\|_{L^c(0,T;L^{a}(\Rr^d))}(1+\left\|\zeta\right\|_{L^{c'}(0,T;L^{a'}(\Rr^d))})\\&\quad+C\left\|g\right\|_{L^c(0,T;L^{a}(\Rr^d))}^\frac{3\tilde{s}+2}{2\tilde{s}}\left(1+\left\|\zeta\right\|_{L^{c'}(0,T;L^{a'}(\Rr^d))}^\frac{2+\tilde{s}}{2\tilde{s}}\right),
\end{align*}which concludes the proof.
\end{proof}

\begin{Lemma}\label{lt3}
Let $d>2$ and assume that $a$ and $c$ satisfy \eqref{te1}.
Let $a'$ and $c'$ be given by \eqref{eq1} and \eqref{eq4}. 
Then, there exists $P,\,Q>1$, $M>c'$, and $0<\beta, \kappa<1$ such that 
\begin{equation}\label{eq9}
\frac{1}{M}=\frac{\beta}{P},
\end{equation}
\begin{equation}\label{eq10}
\frac{1}{a'}=1-\beta+\frac{\beta}{Q},
\end{equation}
\begin{equation}\label{eq11}
\frac{1}{Q}=1-\kappa+\frac{2\kappa}{2^*},
\end{equation}
and
\begin{equation}\label{eq12}
\kappa P\leq 1.
\end{equation}.
\end{Lemma}
\begin{proof}
The result follows from simple computations. It can be checked by using the software \textit{Mathematica}.
\end{proof}

\begin{Corollary}\label{cor1}
Let $(u,m)$ solve \eqref{mfg} and assume that Assumptions A\ref{ham}-A\ref{dxh} are satisfied. Suppose that
$a$ and $c$ satisfy \eqref{te1}.
Then, there exists $C>0$ for which 
$$\int_{\Rr^d}u_\xi(x,\tau)\phi(x)dx\leq C+C\left\|g\right\|_{L^c(0,T;L^{a}(\Rr^d))}^{\theta_1}+C\left\|g\right\|_{L^c(0,T;L^{a}(\Rr^d))}^{\theta_2},$$where $$\theta_1=\frac{1}{1-\kappa\beta},$$and $$\theta_2=\frac{3\tilde{s}+2}{2\tilde{s}}+\frac{\kappa\be(2+\tilde{s})}{2\tilde{s}(1-\kappa\be)}.$$
\end{Corollary}
\begin{proof}
Let $M,P,Q, \beta, \kappa$ as in 
 Lemma \ref{lt3} so that conditions \eqref{eq9}-\eqref{eq12} are simultaneously satisfied. Then,
H\"older's inequality implies
\begin{align*}
\left\|\zeta\right\|_{L^{c'}(0,T;L^{a'}(\Rr^d))}&\leq C\left\|\zeta\right\|_{L^M(0,T;L^{a'}(\Rr^d))}\\&\leq C\left\|\zeta\right\|_{L^\infty(0,T;L^1(\Rr^d))}^{1-\beta}\left\|\zeta\right\|^\beta_{L^P(0,T;L^Q(\Rr^d))},
\end{align*}since \eqref{eq9} and \eqref{eq10} hold.
Because of \eqref{eq11},
we also have
\begin{align*}
\left(\int_{\Rr^d}\zeta^{Q}dx\right)^\frac{1}{Q}&\leq \left(\int_{\Rr^d}\zeta dx\right)^{1-\kappa}\left(\int_{\Rr^d}\zeta^\frac{2^*}{2}dx\right)^\frac{2\kappa}{2^*}\\&\leq C+ C\left(\int_{\Rr^d}\left|D\zeta^\frac{1}{2}\right|^2dx\right)^\kappa,
\end{align*}where the last inequality follows from the Gagliardo-Nirenberg Theorem. By choosing $P$ according to \eqref{eq12}, it follows that
\begin{align*}
\left[\int_0^T\left(\int_{\Rr^d}\zeta^{Q}dx\right)^\frac{P}{Q}dt\right]^\frac{1}{P}&\leq C +C\left[\int_0^T\left(\int_{\Rr^d}\left|D\zeta^\frac{1}{2}\right|^2\right)^{\kappa P}dt\right]^\frac{1}{P}\\&\leq C+C\left[\int_0^T\int_{\Rr^d}\left|D\zeta^\frac{1}{2}\right|^2dxdt\right]^\kappa.
\end{align*}By combining these, we obtain
$$\left\|\zeta\right\|_{L^{c'}(0,T;L^{a'}(\Rr^d))}\leq C+C\left[\left\|g\right\|_{L^c(0,T;L^a(\Rr^d))}\left\|\zeta\right\|_{L^{c'}(0,T;L^{a'}(\Rr^d))}\right]^{\kappa\beta}.$$
Then, 
Young's inequality weighted by $\varepsilon$ yields
$$\left\|\zeta\right\|_{L^{c'}(0,T;L^{a'}(\Rr^d))}\leq C+C\left\|g\right\|_{L^c(0,T;L^a(\Rr^d))}^\frac{\kappa\beta}{1-\kappa\beta}.$$Therefore,
\begin{align*}
\int_{\Rr^d}u_\xi(x,\tau)\phi(x)dx&\leq C+C\left\|g\right\|_{L^c(0,T;L^a(\Rr^d))}^{\frac{1}{1-\kappa\be}}+C\left\|g\right\|_{L^c(0,T;L^a(\Rr^d))}^{\frac{3\tilde{s}+2}{2\tilde{s}}+\frac{\kappa\be(2+\tilde{s})}{2\tilde{s}(1-\kappa\be)}},
\end{align*}which finishes the proof.
\end{proof}

We present the proof of Proposition \ref{proposition1}:
\begin{proof}[Proof of Proposition \ref{proposition1}.]

Proposition \ref{proposition1} follows from Corollary \ref{cor1} by considering the supremum, firstly with respect to $\phi$ and then with respect to $\tau\in[0,T]$. 
\end{proof}
We end now with the proof of  Theorem \ref{mainresult}. 

\begin{Lemma}\label{lt4}
Let $d>2$ and assume that
\begin{equation}\label{te2}
0\,<\,\alpha\,\leq\,\frac{1}{d-1}.
\end{equation}Then, there exist $a$ and $c$ satisfy \eqref{te1} and
\begin{equation}
\label{condalpha}
\alpha c \leq \alpha +1 \,\,\,\mbox{and}\,\,\,\, \alpha a =\frac{2^*(\alpha +1)}{2}.
\end{equation}
\end{Lemma}
\begin{proof}
Once more, the proof relies on elementary computations ans can be checked by recurring to the software \textit{Mathematica}.
\end{proof}

\begin{proof}[Proof of Theorem \ref{mainresult}]
 Since A\ref{alpha} holds, 
Lemma \ref{lt4} ensures the exists of $a$ and $c$ 
satisfying \eqref{te1} and \eqref{condalpha}.
Then, by Lemma \ref{lemma2}, 
we have $g(m)\in L^c([0,T], L^a(\Tt^d))$.
 Thus we can apply Proposition \eqref{proposition1} to obtain the estimate in the Theorem.
 \end{proof}


\bibliography{mfg}
\bibliographystyle{plain}

\end{document}